\documentclass[a4paper,11pt]{article}
\usepackage{amsmath}
\usepackage{amsfonts}
\usepackage{amssymb}
\usepackage{amsthm}
\usepackage{fullpage}
\usepackage{booktabs}
\usepackage{cancel}
\usepackage{color}
\usepackage{xspace}
\usepackage{calc}

\theoremstyle{plain}
\newtheorem{theorem}{Theorem}
\newtheorem{lemma}[theorem]{Lemma}
\newtheorem{corollary}[theorem]{Corollary}
\newtheorem{proposition}[theorem]{Proposition}
\newtheorem{observation}[theorem]{Observation}

\newcommand{\defn}{\emph}
\newcommand{\veqdots}{\mathrel{\makebox[\widthof{=}]{\vdots}}}

\newcommand{\upperchi}{\hat\chi}

\title{The Namer--Claimer game}
\author{Ben Barber\footnote{School of Mathematics, University of Bristol and Heilbronn Institute for Mathematical Research, Bristol, UK.  {\tt b.a.barber@bristol.ac.uk}}
}

\begin{document}

\maketitle

\begin{abstract}
In each round of the Namer--Claimer game, Namer names a distance $d$, then Claimer claims a subset of $[n]$ that does not contain two points that differ by $d$.
Claimer wins once they have claimed sets covering $[n]$.
I show that the length of this game is $\Theta(\log \log n)$ with optimal play from each side.

\end{abstract}

\section{Introduction}\label{intro}

Consider the following game played on $[n] = \{1,\ldots,n\}$ between two players, Namer and Claimer.
In each round, Namer names a distance $d \in \mathbb N$, then Claimer claims a subset of the unclaimed vertices which is `$d$-free'; that is, does not contain two points that differ by $d$.
Claimer wins once they have claimed the entirety of $[n]$.
How quickly can Claimer win with optimal play from both players?

It is not too difficult to prove that the length of the game has an order of growth somewhere between $\log \log n$ and $\log n$ (see Section~\ref{sec:greedy}).
For several years I have been asking at workshops and open problem sessions (including~\cite{Berlin,AC}) which if either of these values is correct.
It turns out that the lower bound is the truth.

\begin{theorem}\label{main}
The length of the Namer--Claimer game on $[n]$ is $\Theta(\log \log n)$.
\end{theorem}

In Section~\ref{proof} I give a direct proof of Theorem~\ref{main}.
In Section~\ref{ramsey} I describe connections with the Ramsey theory of Hilbert cubes.

It is convenient to rephrase the game in graph theoretic language.
For $d \in D=[n-1]$, let $G_d$ be the graph on $V = [n]$ with $xy \in E(G_d)$ if and only if $|x-y|=d$.
For a subset $A$ of $V$, write $G_d[A]$ for the subgraph of $G_d$ induced by $A$.
The \defn{Namer--Claimer game on $\mathcal G = (G_d)_{d \in D}$} proceeds as follows.
\begin{itemize}
\item Set $A_0 = V$.
\item In the $i$th round:
\begin{itemize}
\item Namer names a $d_i \in D$.
\item Claimer selects an independent set $C_i$ of $G_{d_i}[A_{i-1}]$.
\item Set $A_i = A_{i-1} \setminus C_i$.
\item If $A_i = \emptyset$, then the game ends.
      Otherwise, continue to the next round.
\end{itemize}
\end{itemize}

The outcome of this process is a partition of $[n]$ into sets $C_1, \ldots, C_k$ such that each $C_i$ is an independent set in $G_{d_i}$.
As such, the Namer--Claimer game can be viewed as a finite, online version of the following variant colouring problem.
The \defn{upper chromatic number} $\upperchi(\mathcal G)$ of a family of graphs $\mathcal G = (G_d)_{d \in D}$ on $V$ is the least $k$ such that, for every $(d_1,\ldots, d_k) \in D^k$, there is a partition $V = C_1 \cup \cdots \cup C_k$ such that each $C_i$ is an independent set in $G_{d_i}$.
This upper chromatic number was introduced by Greenwell and Johnson~\cite{GJ} for the family of `distance-$d$' graphs on $\mathbb R^n$: $xy$ is an edge of $G_d$ if and only if $\|x-y\|_2=d$.
(The term `upper chromatic number' has also been used by Voloshin~\cite{Voloshin} and later authors for an unrelated concept.)
Greenwell and Johnson proved that $\upperchi(\mathbb R) = 3$.
This remains the only known value of $\upperchi(\mathbb R^n)$; it is not even known whether $\upperchi(\mathbb R^2)$ is finite.
See \cite{Abrams1,Archer,Abrams2} for related results.

The rules above define a Namer--Claimer game for any finite vertex set $V$ and family $\mathcal G$ of graphs on $V$.
One natural instantiation is to take $V$ to be a group $H$ of order $n$, $D = H\setminus\{0\}$ and $G_d = \{\{x,dx\} : x \in H\}$.
The $\Omega(\log \log n)$ lower bound from Proposition~\ref{greedy} is still valid in this setting.
If $H$ is either cyclic or of the form $C_t^m$ with $t$ fixed, then similar methods to those used here show that Claimer has a strategy to end the game in $O(\log \log n)$ rounds.
What happens more generally?

\section{Proof of Theorem~\ref{main}}\label{proof}

We assume freely for convenience that $n$ is larger than some absolute but unspecified constant, but this constant is not particularly large.

\subsection{Greedy bounds}\label{sec:greedy}

To prove an upper bound on the length of the game we must give a strategy for Claimer.
To prove a lower bound we must give a strategy for Namer.
The easiest strategies to analyse are often greedy strategies.

There are various ways `the greedy strategy' for each player could be made precise.
We shall say that
\begin{itemize}
\item the \defn{greedy strategy for Claimer} is to take as many points as possible each round;
\item the \defn{greedy strategy for Namer} is to name the most commonly occurring distance between the set of remaining points in each round.
\end{itemize}

\begin{proposition}\label{greedy}
By playing their respective greedy strategies,

(i) Claimer can win in at most $1 + \log_2 n$ rounds;

(ii) Namer can prolong the game at least $(1+o(1))\log_2 \log_2 n$ rounds.
\end{proposition}

\begin{proof}
(i) 
For any set of unclaimed points $A$ and any $d$, $G_d[A]$ is bipartite, so has an independent set with at least half of the points of $A$.
Thus whenever there are $m$ points remaining, Claimer has a move that will leave at most $\lfloor m/2 \rfloor$ points unclaimed.  
If Claimer always makes such a move, then the game will last at most $1 + \log_2 n$ rounds.

(ii) Suppose that in some round the set of unclaimed points is $A$ with $|A| \geq \alpha n \geq 100$.
These points determine (with multiplicity) $\binom {\alpha n} 2$ distances.
Let $d$ be the most commonly occurring distance.
There are $n-1$ possible distances, so $d$ occurs at least
\[
\frac {\alpha n(\alpha n-1)} {2(n-1)} = \frac {\alpha^2n \left(n-\frac 1 \alpha\right)} {2(n-1)}  \geq \frac{0.99\alpha^2 n} 2
\]
times.
Thus $G_d[A]$ is a vertex-disjoint union of paths with at least $0.99\alpha^2 n/2$ edges in total.
From each path of length $m \geq 1$, Claimer can claim at most $\lceil \frac{m+1} 2 \rceil \leq \tfrac m 2 + 1$ points, so there will be at least $0.99\alpha^2 n/4 \geq \alpha^2 n/5$ unclaimed points going into the next round.
Iterating this argument, after $t$ rounds of their greedy strategy starting from $A$, either fewer than $100$ points remain unclaimed, or Namer has defended at least $5(\tfrac \alpha 5)^{2^t}n$ points from being claimed.
Taking $\alpha = 1$, the game will last at least until
\[
\frac {5}{5^{2^t}}n < 100,
\]
or $t > \log_2 (\log_5 n - \log_5 20) = (1+o(1))\log_2 \log_2 n$.
\end{proof}

The proof of the lower bound is almost identical to Gunderson and R\"odl's proof~\cite{GR} of an upper bound for a certain Ramsey number (see Theorem~\ref{ramsey-bounds} in Section~\ref{ramsey}).

It is tempting to believe that the upper bound could be tight for powers of $2$.
The first counterexample is $n=8$, for which the length of the game is $3$.
(This is easy to check exhaustively on a computer.)
It is instructive to consider a sequence of optimal play from each side.

\newcommand{\claim}[1]{{\color{red}\cancel{#1}}}
\[
\begin{array}{ccccccccc}
 d_1=1 : & \claim{1} & {2} & \claim{3} & {4} & \claim{5} & {6} & {7} & \claim{8} \\
 d_2=2 : & & \claim{2} &  & {4} &  & \claim{6} & \claim{7} & \\
 d_3=1 : & &  &  & \claim{4} &  &  &  & \\
\end{array}
\]
The `natural' move for Claimer in response to Namer's $d_1=1$ is to take all of the odd (or even) numbers, but following this pattern commits to the game lasting $1 + \log_2 n = 4$ turns.
By claiming the less uniform set $\{1,3,5,8\}$ Claimer ensures that there are few repeated distances available for Namer in the next round, allowing them to win sooner.

This suggests that Claimer might do better by playing randomly.
Indeed, if the set of unclaimed points $A$ is always a random set of density $\alpha$ then all distances appear roughly equally often and the analysis in Proposition~\ref{greedy}(ii) is tight.
However, it is not at all clear that this situation can be arranged.
For example, suppose that $d_1 = 1$ and that Claimer claims a random maximal independent set of $G_1$ in the first round.
Then the set of unclaimed vertices $A_1$ will contain more points at distance $2$ than would be expected in a random set of its density.
There are similar dependence problems whenever arithmetic relations hold among the set of named distances.

This makes a truly random strategy for Claimer difficult to analyse.
In the remainder of this section I'll present a hybrid strategy with a small amount of randomness that proves an upper bound matching the lower bound from Proposition~\ref{greedy} up to a multiplicative constant.

\subsection{The lazy strategy}

In the \defn{lazy strategy for Claimer}, if the set of unclaimed points is $A$ and Namer names $d$, then Claimer claims all those $x \in A$ such that $x + d \notin A$.
Equivalently, Claimer claims the largest element of each path component of $G_d[A]$.
This is not a good strategy, but it is very easy to analyse.
Indeed, starting from a set of unclaimed points $A_0$ we have
\begin{align}
A_1 & = A_0 \cap (A_0 - d_1) \nonumber\\
A_2 & = A_1 \cap (A_1 - d_2) = A_0 \cap (A_0 - d_1) \cap (A_0 - d_2) \cap (A_0 - d_1 - d_2)\nonumber\\
    & \veqdots \nonumber\\
A_k & = \bigcap_{I \subseteq [k]} \Big(A_0 - \sum_{i \in I} d_i\Big).\label{evolution}
\end{align}
Hence $x \in A_k$ if and only if 
\[
H(x; d_1, \ldots, d_k) = \Big\{x + \sum_{i \in I} d_i : I \subseteq [k]\Big\} \subseteq A_0.
\]
Call such a set a \defn{(Hilbert) cube of dimension $k$}.
(The `side lengths' $d_i$ of a Hilbert cube are sometimes assumed to be distinct, but we do not assume this here.)
Thus $A_k \neq \emptyset$ if and only if $A_0$ contains a cube of dimension $k$.
Unfortuntately, $[n]$ itself is a cube of dimension $n-1$ (with $d_1=\cdots=d_{n-1}=1$), so the lazy strategy only bounds the length of the game by $n$.
The proof of the upper bound in Theorem~\ref{main} could be viewed as a series of contortions for improving the bound given by the lazy strategy.

\subsection{Introducing randomness}

The first idea is to introduce a small, easily controlled, amount of randomness.
Claimer begins by choosing a uniformly random partition $[n] = A_0 \cup B_0$, then plays the lazy strategy starting from $A_0$ followed by the lazy strategy starting from $B_0$.
The effectiveness of this strategy is determined by the size of the largest cube contained in a random subset of $[n]$.
Unfortunately, a random subset of $[n]$ contains intervals of length $\Theta(\log n)$ with high probability, so this strategy proves at best a logarithmic upper bound on the length of the game.
This would be an improvement on the naive lazy strategy, but not on the greedy upper bound from Proposition~\ref{greedy}.

The second idea is to notice that the cubes causing problems are rather atypical: a cube of dimension $k$ could have up to $2^k$ elements, but the intervals causing problems only have size $k+1$.
We would be much happier if the cubes generated by the named distances had the maximum possible size; that is, they are \defn{non-degenerate}.
This would be true if the named distances were sufficiently well spaced.

\begin{observation}\label{spaced}
Let $d_1, \ldots, d_k \in \mathbb N$ with $d_{i+1} \geq 2d_i$ for $1 \leq i \leq k-1$.
Then every cube $H(x;d_1,\ldots,d_k)$ is non-degenerate.\qed
\end{observation}

Non-degenerate cubes are much less likely to appear in random sets.

\begin{lemma}\label{proper}
Let $k = \lceil \log_2 \log_2 n + 2\log_2 \log_2 \log_2 n\rceil$.
Then there is a partition $[n] = A_0 \cup B_0$ such that neither $A_0$ nor $B_0$ contain a non-degenerate cube of dimension $k$.
\end{lemma}

\begin{proof}
Let $A$ be a subset of $[n]$ chosen uniformly at random and let $B = [n] \setminus A$.
A cube of dimension $k$ is determined by a choice of $x, d_1, \ldots, d_k$, so the expected number of cubes of dimension $k$ and size $2^k$ in $A$ is at most
\begin{align*}
n^{k+1}\left(\tfrac 1 2\right)^{2^k} = 2^{(k+1)\log_2 n - 2^k} & \leq 2^{(2 \log_2 \log_2 n)\log_2 n - (\log_2 \log_2 n)^2\log_2 n} = o(1).
\end{align*}
The same bound holds for $B$, so the probability that either $A$ or $B$ fail to have the desired property is $o(1)$ by Markov's inequality.
\end{proof}

It follows from Observation~\ref{spaced} and Lemma~\ref{proper} that Claimer has a strategy to win in $2k = O(\log \log n)$ rounds if Namer commits to naming sufficiently well-separated distances.
At the other extreme, Namer commits to only ever naming the same distance $1$, say.
In this case Claimer wins in only two turns, as they could claim the odd numbers in the first round and the even numbers in the second.
We would like to say that some combination of these strategies should handle the entire range of intermediate possibilities.
Lemma~\ref{parallel} is one way of making this precise.

\begin{lemma}\label{parallel}
Let $D_1 \cup \cdots \cup D_m$ be a partition of $[n-1]$ with the following property.
Suppose that Claimer has a strategy $\mathcal S$ to win in $k$ rounds if Namer can name at most one distance from each $D_i$, and, for each $i$, a strategy $\mathcal S_i$ to win in $l$ rounds if Namer can name only distances from $D_i$.
Then Claimer has a strategy to win in $kl$ rounds.
\end{lemma}

Crucially, the bound on the performance of the combined strategy does not explicitly depend on $m$.

\begin{proof}
The first time Namer names a distance from $D_i$, Claimer plays according to $\mathcal S$.
Each subsequent time Namer names a distance from $D_i$, Claimer plays according to $\mathcal S_i$.
If the game has not yet ended then Claimer has played at most $k-1$ moves in $\mathcal S$ and at most $l-1$ moves in each of the at most $k-1$ stratgies $\mathcal S_i$ in which they have played any moves at all.
So Claimer wins in at most $(k-1) + (k-1)(l-1) + 1 = (k-1)l+1$ rounds.
\end{proof}

It remains to present sets $D_i$ and strategies $\mathcal S, \mathcal S_i$ to which we can apply Lemma~\ref{parallel}.

\begin{lemma}\label{strategylets}
Let $k = \lceil \log_2 \log_2 n + 2\log_2 \log_2 \log_2 n\rceil$, let $D_1 = \{1,2\}$ and, for $i \geq 2$, let $D_i = [2^{i-1}+1,2^i]$.

(i) If Namer can name at most one distance from each $D_i$, then Claimer has a strategy to win in $4k$ rounds.

(ii) For each $i$, if Namer can name only distances from $D_i$, then Claimer has a strategy to win in $3$ rounds.
\end{lemma}

\begin{proof}
(i)
By Lemma~\ref{proper}, there is a partition $[n] = A_0 \cup B_0$ such that neither $A_0$ nor $B_0$ contains a non-degenerate cube of dimension $k$.

For the first $2k$ rounds of the game, Claimer plays the lazy strategy starting from $A_0$.
Let $d_1 < \cdots < d_{2k}$ be the named distances.
Then by \eqref{evolution},
\[
A_{2k} = \bigcap_{I \subseteq [2k]} \Big(A_0 - \sum_{i \in I} d_i\Big)
 \subseteq \bigcap_{I \subseteq \{2,4,\ldots,2k\}} \Big(A_0 - \sum_{i \in I} d_i\Big).
\]
Thus if there is an $x \in A_{2k}$, then $A_0$ contains a cube $H(x;d_2,d_4,\ldots,d_{2k})$.
Since the $d_i$ are from distinct $D_i$, for $1 \leq i \leq k-1$ we have  $d_{2(i+1)} \geq 2d_{2i}$, whence this cube is non-degenerate by Observation~\ref{spaced}.
But there are no such cubes in $A_0$, so $A_{2k}$ must be empty and Claimer claims all of $A_0$ in the first $2k$ rounds.
Claimer claims all of $B_0$ in the next $2k$ rounds similarly.

(ii)
For each $i$, let $A^i_0 \cup A^i_1 \cup A^i_2$ be the partition of $\mathbb N$ with
\[
A^i_j = \bigcup_{k=0}^\infty \big([2^{i-1}] + (3k+j)2^{i-1}\big).
\]
If $d \in D_i$, then no two points of $A^i_j$ differ by $d$, as all distance between points of $A^i_j$ are at most $2^{i-1}$ or at least $2^i+1$.
Thus Claimer can claim $A^i_0\cap [n]$ in round $1$, $A^i_1\cap [n]$ in round $2$ and $A^i_2\cap [n]$ in round $3$.
\end{proof}

The desired upper bound is a direct application of Lemmas~\ref{parallel} and~\ref{strategylets}.

\begin{corollary}\label{upper-bound}
Claimer has a strategy to win in at most $(1+o(1))12\log_2 \log_2 n$ rounds.\qed
\end{corollary}

\section{Ramsey theory}\label{ramsey}

We might ask whether Lemma~\ref{proper} can be strengthened to provide a partition into pieces that do not contain any Hilbert cubes, degenerate or otherwise.
This is precisely a question of Ramsey theory.
Write $h(k,r)$ for the least $n$ such that whenever $[n]$ is partitioned into $r$ parts, at least one of the parts contains a cube of dimension $d$.

\begin{proposition} \label{use-ramsey}
Let $r\geq 2$ and let $k$ be least with $h(k,r) > n$.
Then Claimer has a strategy to win in at most $rk$ rounds.
\end{proposition}

\begin{proof}
By the choice of $k$ there is a partition $[n] = A_1 \cup \cdots \cup A_r$ such that no $A_i$ contains a cube of dimension $k$.
By following the lazy strategy starting from $A_1$, Claimer can claim all of $A_1$ in at most $k$ rounds.
By repeating for each of $A_2, \ldots, A_r$, Claimer wins in at most $rk$ rounds.
\end{proof}

To apply Proposition~\ref{use-ramsey} we would require lower bounds on $h(k,r)$.
The best known bounds for large $r$ are due to Gunderson and R\"odl.

\begin{theorem}[{\cite[Theorem~2.5]{GR}}]\label{ramsey-bounds}
For $k\geq 3$ and $r \geq 2$,
\[
r^{(1+o(1))\frac {2^k-1} k} \leq h(k,r) \leq (2r)^{2^k-1},
\] 
with $o(1) \to 0$ as $r \to \infty$.
\end{theorem}

In their proof of the lower bound, Gunderson and R\"odl mostly work with non-degenerate cubes.
It is easy to show that a degenerate cube contains an arithmetic progression of length $3$,
so they begin by using their own partitioning variant of Behrend's~\cite{Behrend} construction to partition $[n]$ into a large number of parts that do not contain any arithmetic progressions of length $3$.
As a result, $r \geq \exp(\Omega(\sqrt{\log n}))$, and the bounds obtained from Proposition~\ref{use-ramsey} are extremely poor.

For small $r$, a random $r$-partition along the lines of Lemma~\ref{proper} shows that $h(k,r) \geq r^{\Omega(k)}$.
Nothing better is known.
For the variant problem where we only want to avoid cubes with distinct side lengths this was strenghtened to $r^{\Omega(k^2)}$ by Conlon, Fox and Sudakov, using an inverse Littlewood--Offord theorem of Tao and Vu~\cite{TaoVu} to characterise cubes with far fewer than $2^k$ elements.

For all of the known Ramsey-type results, the major difficulties are presented by non-degenerate cubes.
Corollary~\ref{upper-bound} is stronger than the bound following from Proposition~\ref{use-ramsey} because we were able to avoid considering degenerate cubes at the modest cost of increasing our bound on the length of the game by a factor of $6$.

\newenvironment{acknowledgements} {\renewcommand\abstractname{Acknowledgements}\begin{abstract}} {\end{abstract}}

\begin{acknowledgements}
\noindent I would like to thank James Aaronson, Joshua Erde and Julia Wolf for helpful discussions.
\end{acknowledgements}

\end{document}